\documentclass[12pt]{article}
\usepackage{amssymb}
\usepackage{amsthm}
\usepackage{epsfig}
\usepackage{color}
\usepackage[usenames,dvipsnames]{xcolor}
\usepackage{amssymb}
\usepackage{amsfonts}
\usepackage{amsmath}
\usepackage{bm}
\newcommand{\norm}[1]{\left\Vert#1\right\Vert}
\newtheorem{theorem}{Theorem}
\newtheorem{lemma}[theorem]{Lemma}
\newtheorem{proposition}[theorem]{Proposition}

\newenvironment{definition}[1][Definition]{\begin{trivlist}
\item[\hskip \labelsep {\bfseries #1}]}{\end{trivlist}}

\newenvironment{remark}[1][Remark]{\begin{trivlist}
\item[\hskip \labelsep {\bfseries #1}]}{\end{trivlist}}
\def\cal{\mathcal}

\begin{document}


\title{Optimal strategies for driving a mobile agent  in a ``guidance by repulsion'' model}

\date{}

\maketitle
 \vspace{-2cm} 

\begin{center}
\author{R. Escobedo\textsuperscript{1,2,3,*},
A. Iba\~nez\textsuperscript{2},
E.Zuazua\textsuperscript{4}
}

\

{\it

$^1$CRCA -- Centre de Recherches sur la Cognition Animale,\\
Universit\'e Paul Sabatier, 118 Route de Narbonne, 31062 Toulouse, France.
\\
$^2$BCAM -- Basque Center for Applied Mathematics,
Alda. Mazarredo 14,
48009 Bilbao, Spain.
\\
$^3$AEPA-Euskadi,
Pza. Celestino Mar\'{\i}a del Arenal 14,
48015 Bilbao, Spain.
\\
$^4$Departamento de Matem\'aticas, Universidad Aut\'onoma de Madrid,
Cantoblanco, 28049 Madrid, Spain.
\\
$^*$corresponding author: escobedor@gmail.com
}

\end{center}

\begin{abstract}
We present a {\it guidance by repulsion} model based on a driver-evader interaction where
the driver, assumed to be faster than the evader, follows the evader but cannot be
arbitrarily close to it, and the evader tries to move away from the driver beyond a short
distance.
The key ingredient allowing the driver to guide the evader is that the driver is able to
display a circumvention maneuver around the evader, in such a way that the trajectory of
the evader is modified in the direction of the repulsion that the driver exerts on the
evader.
The evader can thus be driven towards any given target or along a sufficiently smooth
path by controlling a single discrete parameter acting on driver's behavior.
The control parameter serves both to activate/deactivate the circumvention mode and to
select the clockwise/counterclockwise direction of the circumvention maneuver. 
Assuming that the circumvention mode is more expensive than the pursuit mode, and that
the activation of the circumvention mode has a high cost, we formulate an optimal control
problem for the optimal strategy to drive the evader to a given target.
By means of numerical shooting methods, we find the optimal open-loop control which
reduces the number of activations of the circumvention mode to one and which minimizes
the time spent in the active~mode.
Our numerical simulations show that the system is highly sensitive to small variations of
the control function, and that the cost function has a nonlinear regime which contributes
to the complexity of the behavior of the system, so that a general open-loop control
would not be of practical interest.
We then propose a feedback control law that corrects from deviations while preventing
from an excesive use of the circumvention mode, finding numerically that the feedback
law significantly reduces the cost obtained with the open-loop control.

\end{abstract}
\

\smallskip
\noindent \textbf{Keywords:} Guidance by repulsion, Driver-evader agents, Optimal strategies, Feedback control law, Nonlinear dynamics, Numerical simulations


\section{Introduction}
\label{I}
``{\it Follow me}'' (FM) is probably the most natural strategy to solve the guidance
problem, that is, to guide something or somebody along a given trajectory or towards a
specific target, whether physically along geographical paths (streets, roads, buildings),
conceptually (language learning, politics) or even spiritually (religion, social
networks). The success of the FM strategy is based on the effect of the attraction that
the guide (the leader, the driver) exerts on the guided (followers, driven).
Also successfull is the ``{\it Do as I do}'' strategy, based on imitation (of behavior)
or alignment (of opinion), and used in the above mentioned fields, among others (social
learning).

Less expected is the effectiveness of the opposite strategy, namely, ``{\it move away
from me}'' (MA), based on repulsive interactions. That repulsion can serve to
guide something or somebody is shown by nature itself, not only by means of gradient
fields (electromagnetic, temperature or chemical fields), where attraction towards high
densities can be viewed as repulsion from low densities, but also by specifically
repelling targets or agents.
In neural development, axonogenesis takes place by combining attractive and repulsive
guidance, so that the axon growth follows guidance cues presented by chemoattractant
and chemorepellent molecules located in the environment of the
cell~\cite{Tessier-Lavigne-Goodman1996,Rorth2011}. In animal herding, sheepdogs are
used to guide sheep flocks through a repulsive force that dogs exert on
sheep~\cite{CoppingerFeinstein2015}.
Guidance cues can also be magnetic, as in drug
targeting~\cite{SenyeiEtAl-1978,SharmaEtAl2015},
thermotactic, as in sperm guidance~\cite{BahatEtAl2012}, cognitive, as in crowd motion
and traffic flow~\cite{Moussaid-Helbing-Theraulaz2011,Bellomo-Dogbe2011} or in opinion
formation~\cite{MoussaidEtAl2013}, acoustic signals, as in animal alarm calls or
instrumental conditioning~\cite{CoppingerFeinstein2015}, food trail pheromones in
ants~\cite{GarnierEtAl2013} (chemical at a scale larger than the cell), etc.

Attractive and alignment guidance problems are being studied for a long time by means
of agent based
models~\cite{Vicsek2012,Helbing2001,D'Orsogna2006,Gueron1996,Levine2000,MogilnerEtAl2003,
PereaEtAl2009,LukemanEtAl2010,Sprott2009,ZhdankinSprott2010,Chen2006,GaziPassino2003,
GaziPassino2004,Shi-Xie2011,Liu2005,GaziPassino2011,EscobedoEtAl2014}, with attention to
guidance by leadership~\cite{Shen2007,CouzinEtAl-2005,Aureli-Porfiri2010}, and optimal
strategies to minimize time guidance (optimal evacuation times) or distance travelled
by the agents have been found for several systems~\cite{Coron2007,Sontag2013,LenhartWorkman2007,
CaponigroEtAl2013,CaponigroEtAl2015}.

Very recently, in 2015, the optimal strategy for a flocking model to reach a target
point or to follow a given trajectory through attractive and alignment guidance has
been presented~\cite{Borzi-Wongkaew2015}.
In this model, individuals from the flock interact through attractive-repulsive and
alignment forces with the rest of individuals. Interactions are symmetric, except
for one specific individual --the leader-- which exerts on each other individual an
extra attractive force. The result is that the leader is followed by the flock, so that,
by controlling the behavior of the leader, a FM strategy can be used to make the flock
reach a given target or move along a given trajectory.

Repulsive forces in attractive-repulsive models have mostly been considered for
collision avoi\-dance with obstacles or with other agents. Guidance by repulsion has
received much less attention, which is reduced, to our knowledge, to the above mentioned
axonogenesis and animal herding. Also very recently, in 2015, an agent based model has
been introduced to describe a so-called defender-intruder interaction, where repulsion
is used by a defender to expel an intruder away as far as possible from a protected
target~\cite{Wang-Li2015}. The authors in~\cite{Wang-Li2015} find an optimal MA strategy,
which not always consists in approaching the intruder as close as possible, but to simply
drive the intruder away beyond a short distance of security. Repulsion in the
intruder-defender interaction is qualitatively different than in collision avoidance or
in interception~\cite{GhoseEtAl2006,Muro2011}, where the attractive and/or alignment
forces determine (most of) the behavior.

Defender-intruder problems fall into the category of
``conflicting interactions''~\cite{Wang-Li2015}, which are well described by the
classical pursuit-evasion (PE) framework~\cite{Isaacs1965,Nahin2007}. The simplest
scenario for a PE interaction consists of a single pursuer that follows and tries to
capture a single evader that tries to escape to infinite from the
pursuer~\cite{Wang-Li2015}.

Although our interest does not focus on conflicting interactions, we adopt here
the PE framework. The guidance by repulsion can indeed be described with the simple
two-agents PE framework, provided two considerations are taken into account:
first, the guide is not exactly a pursuer, as it often separates from the direction
towards the evader and the guide cannot be arbitrarily close to the evader,
and second, the agent to be guided is not exactly an evader, as it doesn't try
necessarily to escape to infinite but simply moves away a short distance from the
repelling guide.

We present here a guidance by repulsion model based on the two-agents PE framework.
We will refer to the guiding agent as {\it the driver}, which tries to drive the
{\it evader}. The driver thus follows the evader but cannot be arbitrarily close to it.
This is especially interesting if the driver cannot approach the evader or contact
between agents should be avoided (because of chemical reactions, animal conflict, etc).
The evader moves away from the driver but doesn't try to escape beyond a not so large
distance. The driver is of course faster than the evader. At a critical short
distance, the driver can display a circumvention maneuver around the evader that forces
the evader to change the direction of its motion. Thus, by adjusting the onset and
offset of the circumvention maneuver, the evader can be driven towards a desired target
or along a given trajectory. Our goal is to find optimal strategies to drive the
evader in the most efficient way.

We use an inertial model where interactions between agents take place through asymmetric
newtonian forces. The asymmetry consists in that one agent is attracted and repulsed,
while the other is simply repulsed. This kind of interaction has been coined
``anti-newtonian'' by Sprott~\cite{Sprott2009} and others~\cite{ZhdankinSprott2010,Wang-Li2015}.
Thus, velocities are not constant (they depend on the state of the system), and no
alignment forces are considered.

We denote by $D$ and $E$ (and indexes $d$ and $e$) the driver and the evader
agents respectively. Both agents obeys the Newton's second law, that is,
$\dot{\vec{u}}_i(t) = \vec{v}_i(t)$ and $m_i \dot{\vec{v}}_i(t) = \vec{\cal F}_i(t)$,
where
$\vec{u}_i(t) = (u_i^x(t),u_i^y(t)) \in \mathbb{R}^2$, 
$\vec{v}_i(t) = (v_i^x(t),v_i^y(t)) =  \in \mathbb{R}^2$, $m_i \in \mathbb{R}^+$ and
$\vec{\cal F}_i(t) = ({\cal F}_i^x(t),{\cal F}_i^y(t)) \in \mathbb{R}^2$ denote the
position vector, the velocity vector and the mass of agent~$i$, and the resultant force
to which agent~$i$ is subject, respectively, for $i=d,e$.

The force acting on the evader $\vec{\cal F}_e(t)$ has only one component, which is in
the direction of escape from the pursuer, $\vec{u}_d - \vec{u}_e$. The force acting on
the pursuer $\vec{\cal F}_d(t)$ has a component collinear to $\vec{\cal F}_e(t)$ and a
lateral perpendicular component $(\vec{u}_d - \vec{u}_e)^\bot$ which allows the pursuer
to surround the evader, therefore forcing the evader to change the direction of escape.
Here $(x,y)^\bot = (-y,x)$.

The perpendicular component of the force acting on the pursuer $\vec{\cal F}_d(t)$ can
be switched on and off by means of the control parameter $\kappa(t)$ which takes values
on $\{-1,0,1\}$. The control parameter $\kappa(t)$ is the key ingredient of the model,
as it determines the behavior of the pursuer, which in turn determines the behavior of
the evader. The resulting dynamical system can be considered as a driver-evader system
with two operating modes controlled by a single parameter.

\paragraph{Contents of the paper and sketch of the results}

Section titles are self-explanatory.

In Sec.~2, we introduce the model equations and parameters, describing in detail the
interactions between agents and with the environment. We show that the driver-evader
system can be viewed as having two operating modes controlled by a single parameter,
so that the system can be moved from one state to another in order to make the driver
to guide the evader from any point to any other point. We study the controllability
of the system, and prove that, although the system is not fully controllable, agents
positions remain asymptotically close to each other ({\it i.e.}, an agent can not go
to infinity).

In Sec.~3, we consider what are the optimal strategies which allow the driver to guide
the evader to a desired target. Activating a system has a cost, as well as keeping it
in the active mode. As an illustrating example, the driver can be viewed as a spacecraft
with two lateral propellers whose ignition process and fuel consumption are very high
with respect to the consumption of the back propeller. Our interest is in reducing the
cost by minimizing 1) the number of activation and 2) the time of use of the lateral
propellers.

To do that, we formulate an optimal control problem $(OCP)$  for a cost functional
including these two costs. Whe then find, in Sec.~3.1, the (unique) optimal strategy
allowing the reach the target in an initially active system ($\kappa(t_0)=1)$, and the
(unique) optimal strategy which reduces the length of the time interval
$[t_{\tt ON}, t_{\tt OFF}]$ during which the control parameter has a non-zero value.
This second strategy also reduces the number of manipulations of the control to two 
(one {\it switch on} and one {\it switch off}). In this minimization process, the cost of
back propellers (which are always switched on, the driver is a self-propelled agent which
is always attracted by the evader) is neglected compared to the cost of using the lateral
propellers, so that there is not (too much) concern by long trajectories or long
execution times where lateral propellers are switched off.

These open-loop controls are subject to the reproducibility of the initial and
environmental conditions and to uncertainty about the model~\cite{Coron2007,Sontag2013}.
In fact, we show here in detail that the system is highly sensitive to small variations
of the values of the control, thus suggesting the appropriateness of the use of
closed-loop controls that can afford for the random perturbations arising in real
systems.

We then present in Sec.~3.2 a feedback control law allowing to drive the evader to
any desired target with an arbitrary accuracy. This feedback law is especially advantageous as it
yields a similar cost than the open-loop controls, in real conditions ({\it i.e.}, under
perturbations). Moreover, the feedback control law provides an excellent insight for the
search of a control function yielding a substantially lower cost; as an example, we
report numerical simulations of a case in which the resulting cost is reduced almost a
60\% of the cost provided by the open-loop controls.

Finally, in Sec.~4, we present our conclusions and we comment on future directions.
 in the last
section
\section{Model formulation and first analysis}
\label{II}

\subsection{Equations and parameters}
The system consists of 8 ordinary differential equations (ODEs) and 8 initial conditions
for the components $u_i^x(t)$, $u_i^y(t)$, $v_i^x(t)$ and $v_i^y(t)$, for $i=d,e$.
Expanding the resultant force acting on the $i$th-agent
$\vec{\cal F}_i(\vec{u}_d(t), \vec{u}_e(t), \vec{v}_d(t), \vec{v}_e(t))$, the system can
be written as the following 4 vectorial ODEs with 4 initial conditions:
\begin{align}
\dot{\vec{u}}_d(t) & = \vec{v}_d(t),
\label{Eq-up}
\\
\dot{\vec{u}}_e(t) & = \vec{v}_e(t),
\\
\dot{\vec{v}}_d(t) & = {1 \over m_d}
 \left[
 - C^E_D {\vec{u}_d(t) - \vec{u}_e(t) \over \|\vec{u}_d(t)-\vec{u}_e(t)\|^2}
 \left( 1 - {\delta_c^2 \over \|\vec{u}_d(t)-\vec{u}_e(t)\|^2} \right)
 \right.
\nonumber
\\
&  \hspace{1cm} 
 \left.
 - C_{\rm R} {\delta_1^4  \over \|\vec{u}_d(t)-\vec{u}_e(t)\|^4}
 \left( \vec{u}_d(t) - \vec{u}_e(t)  - \kappa(t) \delta_2 \,
 {(\vec{u}_d - \vec{u}_e)^\bot \over \| \vec{u}_e(t) - \vec{u}_d(t)\|} \right)
 - \nu_d \vec{v}_d(t) \right],
 \label{Eq-vp}
\\
\dot{\vec{v}}_e(t) & = {1 \over m_e}
 \left[ C^D_E {\vec{u}_e(t) - \vec{u}_d(t) \over \|\vec{u}_e(t)-\vec{u}_d(t)\|^2}
  - \nu_e \vec{v}_e(t) \right],
\label{Eq-vs}
\\
\vec{u}_d(t_0) & = \vec{u}_d^0, \quad \vec{u}_e(t_0) = \vec{u}_e^0, \quad
\vec{v}_d(t_0) = 0 \quad \mbox{and} \quad \vec{v}_e(t_0) = 0.
\label{CondInis}
\end{align}
Interactions between agents are as follows.

The expression between brackets in Eq.~(\ref{Eq-vp}) consists of two terms with
respective coefficients $C_D^E$ and $C_R$. The first term corresponds to the long-range
attraction and short-range repulsion force that the evader exerts on the driver. Here
$\delta_c$ is the distance at which the attraction balances the repulsion:
if $\|\vec{u}_e(t)-\vec{u}_d(t)\| > \delta_c$, then the evader attracts the driver (and
therefore the driver accelerates towards the evader), while if
$\|\vec{u}_e(t)-\vec{u}_d(t)\| < \delta_c$, then the driver is repulsed by the evader
(and the driver decelerates in the direction opposed to the evader).

The term with coefficient $C_R$ corresponds to the circumvention force, which as a
component collinear to the attraction-repulsion interaction with the evader, given by
$\vec{u}_d - \vec{u}_e$, and a component perpendicular to this direction, denoted by
$(\vec{u}_d - \vec{u}_e)^\bot$, where $(x,y)^\bot = (-y,x)$. The circumvention force
is thus a force that attracts the driver towards one of the two sides of the evader,
where the side is determined by the sign of the parameter $\kappa(t)$.

The model considers two other critical distances, $\delta_1$, which is the (short)
distance at which the intensity of the circumvention force is effective, and $\delta_2$,
which denotes the distance between agents left by the driver during the circumvention
maneuver.
In~(\ref{Eq-vs}), the evader is simply subject to the repulsion from the driver,
which has the same expression than the force that the evader exerts on the driver but
with a different (and smaller) coefficient $C^D_E$.

Finally, both agents are subject to friction forces with the ground, which have the
same form, in the opposite direction of the motion: $-\nu_i \vec{v_i}$, for $i=d,e$.

The model can be formulated in different ways. The formulation presented
in~(\ref{Eq-up})--(\ref{CondInis}) is based on the general expressions of the
attraction-repulsion forces introduced by
Gazi \& Passino~\cite{GaziPassino2003,GaziPassino2004}, later widely used in realistic
(biological) models~\cite{Liu2005,ZhdankinSprott2010,Shi-Xie2011,GaziPassino2011,
EscobedoEtAl2014}.
Agents are prevented from collisions by means of strong short-range repulsive forces
(large exponents of ${\|\vec{u}_e(t)-\vec{u}_d(t)\|}$ in the denominator), so that,
although it is theoretically possible, in practice, the denominators are never (too
close to) zero,
and the model is in this sense well-posed. 
Another advantage of this formulation is that each ingredient of the model, especially
the circumvention behavior, appears explicitly in the equations, thus facilitating the
realistic interpretation of the contribution of each component of the model. For example,
the exponent~$-4$ of $\|\vec{u}_d-\vec{u}_e\|$ in the circumvention force (which is
large compared to the exponent~$-2$ in the attracting force) shows that the circumvention
maneuver will only be effective when both agents are sufficiently close to each other.

With this formulation, the control parameter $\kappa(t)$ appears as a factor of the
perpendicular component of the force acting on the evader, and, by the simple choice
of a value in $\{-1,0,1\}$, practically determines the behavior of the system.

\subsection{Two operating modes, a single control parameter}
Let us define the instantaneous distance between agents
$r(t) = \|\vec{u}_d(t)-\vec{u}_e(t)\|$.

Then, omitting time-dependence to lighten notation, Eqs.~(\ref{Eq-vp})--(\ref{Eq-vs}) can
be rewritten as
\begin{align}
\dot{\vec{v}}_d & = - {C^E_D \over m_d} {\vec{u}_d - \vec{u}_e \over r^2}
 \left[1 + {1 \over r^2}
 \left({C_{\rm R} \over C^E_D} \delta_1^4 - \delta_c^2 \right)  \right]
+ \kappa {\delta_1^4 \delta_2 \over r^3}
 {C_{\rm R} \over m_d} \, {(\vec{u}_d - \vec{u}_e)^\bot \over r^2}
 - {\nu_d \over m_d} \vec{v}_d,
\label{Eq-vp-d}
\\
\dot{\vec{v}}_e & = - {C^D_E \over m_e} {\vec{u}_d - \vec{u}_e \over r^2}
  - {\nu_e \over m_e} \vec{v}_e.
\label{Eq-ve-d}
\end{align}

We have solved the system~(\ref{Eq-up})-(\ref{CondInis}) numerically with several
classical methods (Runge-Kutta, adaptive or not, Crank-Nicolson), finding that a simple
explicit Euler method with time-step of order $10^{-6}$ is sufficiently effective and
provides a sufficiently accurate solution for our purposes.
Our results are based on exhaustive numerical simulations for a wide range of the
parameters preserving the significance of the model.
The behavior of the model is stable and coherent under reasonable variations of the
parameters. The results we describe here are not dataset dependent. The values we have
chosen are those providing the more illustrative figures.

Eq.~(\ref{Eq-vp-d}) shows that a necessary condition to have an effective short-range
repulsion acting on the driver is therefore that $C_{\rm R}\delta_1^4 - C^E_D \delta_c^2$
is negative, as it is the case for the dataset considered in our study. Similarly, in
order to have a faster and more reactive behavior of the driver, we assume that
$C^E_D > C^D_E$ and that $m_d$ and $\nu_d$ are sufficiently smaller than $m_e$ and
$\nu_e$ respectively.

For the numerical simulations and graphical descriptions presented here, we have
considered the following dataset for the typical values of the constant parameters of the
model: first, $m_i$ and $\nu_i$ are the mass and the friction coefficients of agent $i$
respectively, with $m_d = 0.4$, $m_e= 1$, $\nu_d = 1$ and $\nu_e = 2$;
then,  
$C_D^E = 3$ is the coefficient of the attractive-repulsive force that the evader exerts
on the driver,
$C_E^D = 2$ is the coefficient of the repulsive force that the driver exerts on the
evader, and
$C_{\rm R} = 0.5$ is the coefficient of the circumvention component of the force exerted
on the driver when $\kappa = \pm 1$.
Finally, the critical distances are as follows:
$\delta_c = 2$,
$\delta_1 = 2$
and $\delta_2 = 2$.

Parameter values are dimensionless (of order one) and chosen so that their relative
value allow the system to reproduce the behavior of a realistic driving phenomenon.
\vspace{0.3cm}

The numerical simulations allow us to establish that the system has two operating
modes, depending on the value of $\kappa$:
\begin{itemize}
 \item {\it Pursuit mode} ($\kappa = 0$): the driver $D$ pursues the evader $E$, which
 moves away from $D$. Both agents tend to move in the same direction, given by their
 acceleration vectors, along the straight line $\overline{DE}$.
 If $\kappa=0$ continuously during a sufficiently long interval of time,
 both agents' velocities converge asymptotically to the same vector of norm $v_{\rm as}$
 (due to the friction), and the driver stays at a constant distance from the evader,
 $\delta_{\rm as}$.
 See Appendix~A for an analytical estimate of these values, and the left panels in
 Fig.~\ref{modes}: the lower panel shows that, before reaching the constant value
 $v_{\rm as} \approx 0.71$, the pursuer is faster than the evader. For the values we
 have used, $\delta_{\rm as} = \sqrt{2}.$

 \item {\it Circumvention mode} ($\kappa=\pm 1$): the driver $D$ separates from the
 straight line $\overline{DE}$ and starts a circumvention maneuver around the evader~$E$
 (clockwise or counterclockwise, depending on the sign of $\kappa$).
 The response of the evader is to move away from the driver. If $\kappa=1$ continuously
 during a sufficiently long interval of time, as the driver is faster than  the evader,
 the asymptotic behavior of the system is a circular motion of the evader  around a fixed
 point and a circular motion of the driver around the circle described by the evader.
 See Fig.~\ref{modes}, where the right panels show that, for
 $t \in [0,80]$, the system tends to a periodic configuration where both  agents are
 separated by a constant distance $\delta_{\rm as}^{\rm ang} \approx 1.82$ (not  shown in
 the figure) and have the same angular velocity of norm  $\omega_{\rm as} \approx \pi/4$
 (the period $s$ of the oscillations is $s = 2\pi/\omega_{\rm as}$; bottom-right
 panel of Fig.~\ref{modes} shows that $s\approx 8$).
\end{itemize}
\begin{figure}[ht]
\vspace{-0.3cm}
\begin{center}
\epsfxsize=0.95\textwidth
\epsfbox{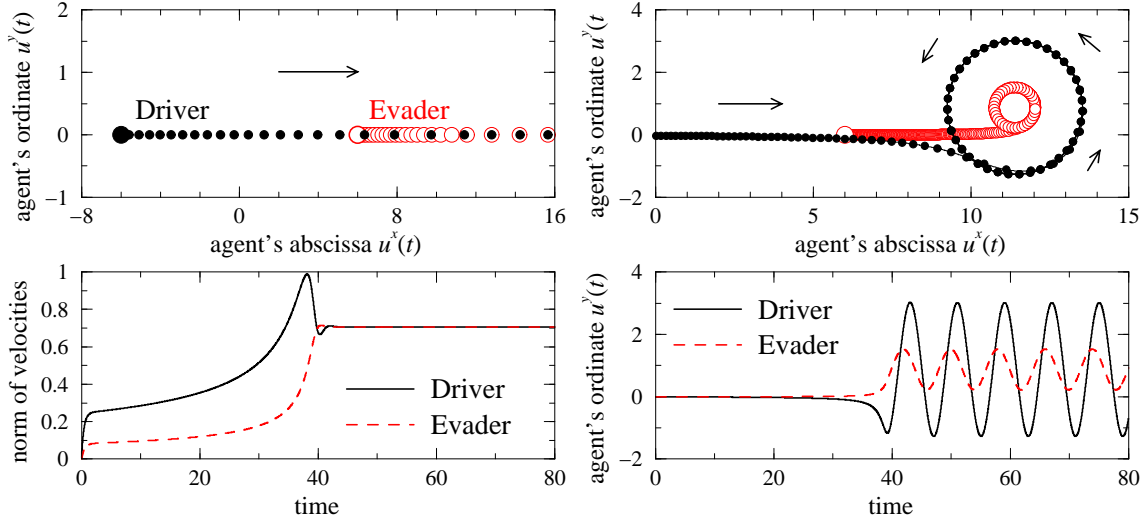} \hspace{1cm}
\vspace{-0.8cm}
\caption{The two operating modes of the system.
Left panels: pursuit mode ($\kappa = 0$). Right panels: circumvention mode ($\kappa =1$).
Initial configuration, in both cases: $\vec{u}_d = (-6,0)$, $\vec{u}_e =(6,0)$ and
$\vec{u}_T = (1,1)$, with zero initial velocities.
Upper panels: agents' trajectories. Arrows denote the direction of motion of the agents.
For sufficiently long times, the velocities converge asymptotically to a constant value
$v_s$.
When $\kappa=1$, the driver turns counterclockwise around the evader, which also turns
counterclockwise.
Left-bottom panel: time variation of the norm of the velocities $\|\vec{v}_{d,e}(t)\|$,
both reaching $v_{\rm as} \approx 0.71$ at $t \approx 38.9$.
Right-bottom panel: time variation of agents' ordinates: $u_d^y(t)<0.1$ until $t\approx 32$,
while $u_e^y(t) < 0.1$ until $t\approx 37$. For $t > 40$, periodic behavior
(of period $s \approx 8$) with constant angular velocity $\omega_{\rm as} \approx \pi/4$
and constant separation $\delta_{\rm as}^{\rm ang} \approx 1.82$.}
\label{modes}
\vspace{-0.3cm}
\end{center}
\end{figure}

When the driver is sufficiently close to the evader, the circumvention mode is effective
and triggers the circular behavior of the agents.
See the right panels of Fig.~\ref{modes}, where the driver initial position at $(-6,0)$
is so far from the initial position of the evader $(6,0)$ that, although the control is
set to $\kappa =1$ from $t_0=0$ to $t_f=100$, the first part of the trajectory is almost
a straight line. Until $t \approx 32$, both agents are almost still in the
horizontal axis: $\vec{u}_d = (5,-0.09)$, $\vec{u}_e = (9.7,0.02)$. At $t \approx 37$,
$\vec{u}_d = (8.62,-0.37)$ and $\vec{u}_e = (11.05,0.1)$, the driver is close enough to
the evader and the circular behavior becomes perceptible. See the oscillations of
$u_d^y(t)$ and $u_e^y(t)$, of period $s \approx 40/5=8$, in the right-bottom panel of
Fig.~\ref{modes}.

\begin{remark}
\label{Remark1}
{\bf 1:} When the driver is far from the evader ({\it i.e.}, $r \gg 1$), the
term between brackets in Eq.~(\ref{Eq-vp-d}) is such that $1 \gg 1/r^2$, so the first
term of this equation (which has coefficient $C_D^E$) is of order ${\cal O}(r^{-1})$,
while the term with $\kappa(t)$ is of order ${\cal O}(r^{-4})$.
Then, Eq.~(\ref{Eq-vp-d}) can be reduced to
\begin{align}
\dot{\vec{v}}_d & = -  {C^E_D \over m_d} {\vec{u}_d - \vec{u}_e \over r^2}
 - {\nu_d \over m_d} \vec{v}_d.
\end{align}
\end{remark}

\begin{remark}
\label{Remark2}
{\bf 2:} In particular, it can be observed that, when $r^3 \gg \delta_2$, then
\begin{align}
 \left\| {C_R \delta_1^4          \over m_d} { \vec{u}_d - \vec{u}_e       \over r^2}
 \right\| = {C_R \delta_1^4 \over m_d r}
 \gg
 \left\| {C_R \delta_1^4 \delta_2 \over m_d} {(\vec{u}_d - \vec{u}_e)^\bot \over r^5}
 \right\| = {C_R \delta_1^4 \delta_2 \over m_d r^4},
\end{align}
so that the term deviating the driver from the pure pursuit trajectory is negligible
with respect to the term corresponding to the attracting force exerted by the evader.
Therefore, when the driver is sufficiently far from the evader, the value of
$\kappa(t)$ has no influence on the behavior of the driver, meaning that $\kappa$ can
be set to zero.
\end{remark}

\subsection{On the controllability of the system}
The circumvention mode can be viewed as the {\it active} state of the system, where the
control \mbox{parameter} $\kappa$ is set to {\tt ON}, while the pursuit mode is the
{\it rest} state of the system, where $\kappa$ is set to {\tt OFF}.
With the appropriate combination of both modes, the driver is able to make the evader
reach any given target point or move along any (relatively smooth) given trajectory.
The resulting behavior of such combination is what we call a {\it driving} behavior.
See Fig.~\ref{tray}.
\begin{figure}[ht]
\vspace{-0.3cm}
\begin{center}
\epsfxsize=0.95\textwidth
\epsfbox{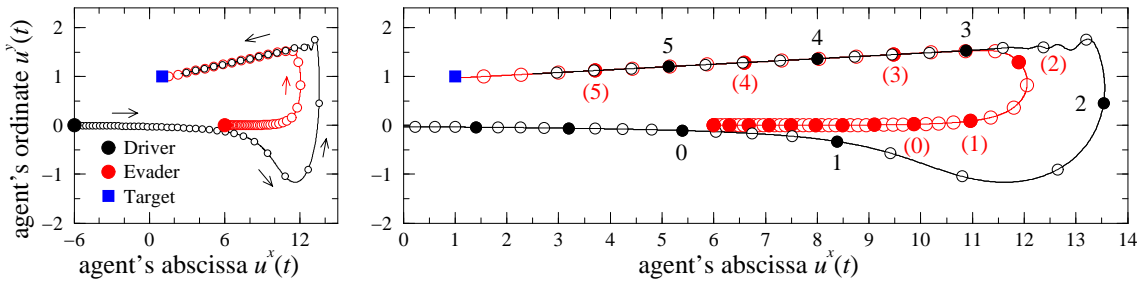} \hspace{1cm}
\vspace{-0.3cm}
\caption{Agents' trajectories (driver: small black circles; evader: large red circles)
for an admissible control function $\kappa(t)$.
Left: whole trajectories, right: zoom for $u^x(t) \in [0,15]$.
Symbols are equidistant in time; filled symbols labeled with
the same number denote same instant of time (evader's labels are shown in parentheses).}
\label{tray}
\end{center}
\vspace{-0.5cm}
\end{figure}

\noindent
This strongly suggests that the system can be controlled with a single control
parameter, $\kappa(t)$. In control theory, a system is said to be {\it fully controllable}
when, starting from any arbitrary initial state, every possible state of the system can
be reached by appropriately adjusting the control parameters
(see, {\it e.g.},~\cite{Trelat2005}).
In this sense, our numerical simulations have shown that agents cannot be separated an
arbitrarily long distance, so that the driver-evader system cannot be labelled as fully
controllable.

In fact, we prove analytically in Appendix~A that, for any initial state, the separation
between agents is bounded by above and tends to a distance (of order 1), showing that the
evader cannot escape from the driver to infinity, and the driver cannot move away from the
evader to infinity.
Moreover, numerical simulations show that agents' velocities are also bounded and that,
when $\kappa(t)$ remains unchanged for sufficiently long time intervals, both velocities
tend asymptotically to the same constant value ($v_{\rm as}$ or $\omega_{\rm as}$).
Finally, agents are prevented from occupying the same place at the same time ({\it i.e.}, 
$\vec{u}_d(t) = \vec{u}_e(t)$) by the strong short-range repulsion force that the driver
exerts on the evader, as our numerical simulations confirm.

We will say instead that the driver-evader system is {\it partially controllable}, in the
sense that each agent can be controlled separately: the driver can force the evader to
reach any point in the plane (this is shown in the successive sections), and vice-versa,
a series of targets for the evader can be selected so that the driver is driven to reach
any point in the plane.

The question arises now as how the driver-evader system behavior can be optimized to
minimize a given cost functional accounting for the use of the lateral propellers.

\section{Optimal control: two optimal open-loop controls, one feedback control law}
\label{III}

Denote by $B_\rho(T)$ a ball of radius $\rho$ centered in $T$, assume that the evader is
initially far from the driver
({\it e.g.}, $r(t_0)=\|\vec{u}_d(t_0)-\vec{u}_e(t_0)\| \gg \delta_1$), and assume also
that $D$, $E$ and the ball $B_\rho(T)$ are not aligned, that is,
$\overline{DE} \cap B_\rho(T) = \O$.

Consider now the objective of driving the evader $E$ into the target ball $B_\rho(T)$ at
a final time~$t_f$ by controlling the driver $D$ with an appropriate strategy~$\kappa(t)$.
As $D$, $E$ and $B_\rho(T)$ are not aligned, the circumvention mode has to be activated
and/or has to remain active a suitable time to modify the trajectory of the evader and
guide it towards the target ball $B_\rho(T)$.

Such an objective must appraise the cost of 1) forcing the system to leave its resting
state, and/or 2) keeping the system in an active state. This cost is given by the
functional
\begin{align}
 J(\kappa) \stackrel{def}{=} \sigma_1 N_{\rm ig}(\kappa) + \sigma_2 \, {\cal C}(\kappa),
\end{align}
where $N_{\rm ig}(\kappa)$ is the number of times that the system is forced to leave its
resting state in the time interval $[t_0,t_f]$ (that is, the number of ignition
processes, where $\kappa(t)$ changes from 0 to $\pm 1$), and $C(\kappa)$ is the total
time during which the system is active in $[t_0,t_f]$ ({\it i.e.}, the time spent with
a lateral propeller in active mode, where $\kappa(t)$ has a non zero value),
\begin{align}
 {\cal C}(\kappa) = \int_{t_0}^{t_f} | \kappa(t) | dt,
\end{align}
and $\sigma_{1,2}$ are nonnegative weights fixed to balance the contribution of each
partial cost.

An optimal control problem $(OCP)$ can then be formulated as
\begin{align}
 (OCP) \; \left\{
 \begin{array}{l}
  {\rm Min} \, J(\kappa) = \sigma_1 N_{\rm ig}(\kappa) + \sigma_2 \, {\cal C}(\kappa) \\
  \kappa \in U_{\rm ad} = \big\{ \kappa : [t_0,t_f] \to \{-1,0,1\} \mbox{ such that }
 \vec{u}_e(t_f) \in B_\rho(T) \big\}  
 \end{array},
 \right.
\end{align}
where $U_{\rm ad}$ is the set of admissible controls.

Here we solve $(OCP)$ for $N_{\rm ig}=0$ and $N_{\rm ig}=1$ in sections 3.1 and 3.2
respectively, finding the corresponding optimal open-loop controls. Our results show that
the system is highly sensitive to small variations of the conditions of the problem,
so that a general open-loop control for $N_{\rm ig} > 1$ would not be of practical
interest. Instead, we provide in section 3.3 a feedback control law for $N_{\rm ig} > 1$
that substantially reduces the cost of the open-loop controls for $N_{\rm ig} \le 1$,
and preserves the number of ignitions at a relatively low value ($N_{\rm ig}=4$; see
later).

\subsection{Optimal open-loop guidance strategies for $N_{\rm ig} \le 1$}

Solving $(OCP)$ for $N_{\rm ig}=0$ consists in minimizing ${\cal C}(\kappa)$ for
$\kappa(t_0) = \kappa_0$, with $\kappa_0 = \pm 1$. Indeed, if $\kappa(t_0)=0$, then
$\kappa \notin U_{\rm ad}$, because $D$ and $E$ are initially not aligned with the target
ball and no inition process can be used to modify the trajectory of the evader, so
$\vec{u}_e(t) \notin B_\rho(T)$ for all $t \in [t_0,+\infty)$.
Similarly, solving $(OCP)$ for $N_{\rm ig}=1$ consists in minimizing ${\cal C}(\kappa)$
for $\kappa(t)$ of the form $\kappa(t) = \kappa_0$ in a time interval
$[t_{\tt ON},t_{\tt OFF}]$ and $\kappa(t) = 0$ outside.

We thus seek two control function profiles like those sketched in Fig.~\ref{shapes}. 

\begin{figure}[ht]
\begin{center}
       \epsfxsize=0.95\textwidth
       \epsfbox{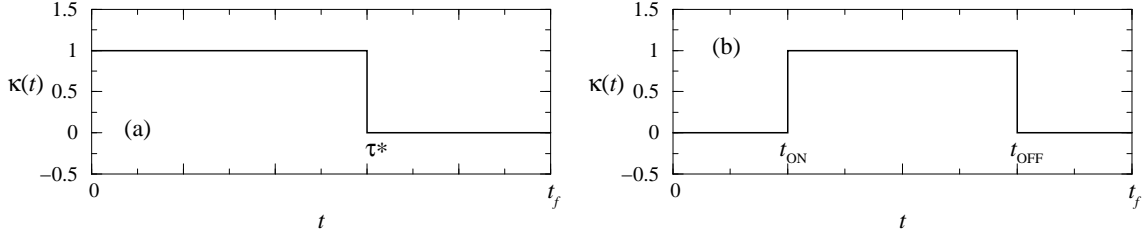} \hspace{1cm}
\vspace{-0.4cm}
\caption{Profiles of the control function $\kappa(t)$ with $\kappa_0 = 1$ for each
case (a) $N_{\rm ig}=0$ and (b) $N_{\rm ig}=1$.}
\vspace{-0.5cm}
\label{shapes}
\end{center}
\end{figure}

\subsubsection{$N_{\rm ig} = 0$:}
The following result holds:

Let $\kappa_\tau(t)$: $\mathbb{R} \to \{-1,0,1\}$ be the following step-function
in the time interval $[t_0,t_f]$,
\begin{align}
 \kappa_\tau(t) = \left\{ 
 \begin{array}{ll}
  \kappa_0 & \mbox{if } t < \tau \\
  0 & \mbox{if } t \ge \tau
 \end{array},
\right.
\label{estecost}
\end{align}
where $\kappa_0 = \pm 1$ is the initial value at time $t_0$: $\kappa_\tau(t_0)=\kappa_0$.
Then, if $t_f$ is sufficiently large, there exists an interval
$[\tau_\alpha, \tau_\omega] \subset (t_0,t_f)$ such that, for all
$\tau \in [\tau_\alpha, \tau_\omega]$, there exists a time $t \in (t_0,t_f)$ for which
the evader is in the interior of the ball of radius $r$ centered in the target $T$.
That is:
\begin{align}
\forall \; \tau \in [\tau_\alpha, \tau_\omega], \; \exists \; t \in (t_0,t_f)
\mbox{ such that } \| \vec{u}_e(t) - \vec{u}_T \| < r.
\end{align}
Moreover, if $r \to 0$, then the interval $[\tau_\alpha, \tau_\omega]$
shrinks to a single point $\tau^*$ such that there exists a time $t \in (t_0,t_f)$
for which $\vec{u}_e(t) = \vec{u}_T$. See $\tau^*$ in Fig.~\ref{shapes}(a).

The proof is based on a continuity argument applied to a (numerical) shooting method.

Fig.~\ref{shoot-1change} shows the trajectories of the evader for different values of
$\tau$. By a continuity argument, it is possible to find the values of $\tau_\alpha$,
$\tau_\omega$ and $\tau^*$ with a simple shooting method based on comparing the
direction of the velocity vector of the evader $\vec{v}_e$ with respect to the direction
towards the target when the control is switched off. See Appendix~B for a more detailed
description of the decision test of the shooting method.

\begin{figure}[ht]
\vspace{-0.2cm}
\begin{center}
       \epsfxsize=0.85\textwidth
       \epsfbox{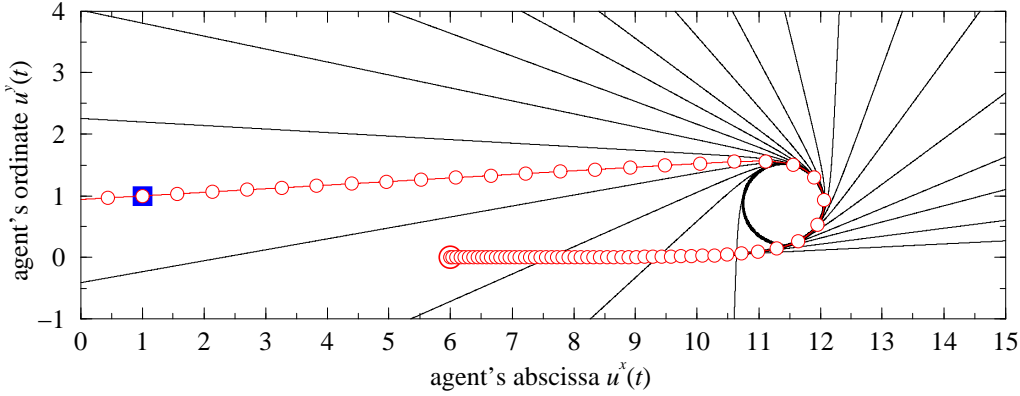} \hspace{1cm}
\vspace{-0.6cm}
\caption{Shooting method for finding the optimal value $\tau^*$ of the control function
shown in Fig.~\ref{shapes}(a), for which $\| \vec{u}_e(t) - \vec{u}_T \| < \rho$ for some
$t \in (0,t_f)$ (here we used $\rho=10^{-4}$).
Filled (blue) square denotes target's position $\vec{u}_T=(1,1)$, empty (red) circle
evader's initial position $\vec{u}_e(0)=(6,0)$; driver's initial position
$\vec{u}_d(0)=(-6,0)$ is not depicted.
Line with (red) circles denote the trajectory of the evader for the optimal value
$\tau^* = 41.15$.
Wide~(black) line denotes the accumulation circle around which the evader turns if the
control is kept to one.
Thin~(black) lines denote trajectories of the evader for the following non-equispaced
values of $\tau$:
34, 36, 37, 37.5, 38, 38.4, 38.7, 39, 39.3, 39.6, 39.9, 40.2, 40.4, 40.6, 40.8, 41, 
41.3, 41.6, 42 and 43.
}
\label{shoot-1change}
\end{center}
\end{figure}

In this case, the cost of a strategy $\kappa_\tau(t)$ given by~(\ref{estecost}) is
${\cal C}(\kappa_\tau) = \tau$ for all $\tau \in [\tau_\alpha, \tau_\omega]$, so that the
optimal strategy is $\kappa_{\tau_\alpha}(t)$. For a sufficiently small value of $r$ (we
used $r=10^{-4}$ in Fig.~\ref{shoot-1change}), the interval $[\tau_\alpha, \tau_\omega]$
colapses to the value $\tau^*$ and the cost of the optimal control is
${\cal C}(\kappa_\tau^*) = \tau^*$.

Noticeably, Fig.~\ref{shoot-1change} reveals that the system is highly sensitive to small
variations of $\tau$ around the optimal value $\tau^*$: the same variation of $\tau$
produces a larger variation of the deviation of $\vec{v}_e$ with respect to a reference
line ({\it e.g.}, the horizontal line) when $\tau$ is close to $\tau^*$ than when $\tau$
is far from $\tau^*$: for $\Delta \tau = 2$, the angular variation $\theta$ from
$\tau = 39.9$ to 42 is larger than $\pi/2$ (actually,
$\theta_{42}-\theta_{39.91} \approx 1.62$), and is
10 times smaller from $\tau = 34$ to 36 ($\theta_{36}-\theta_{34} \approx 0.16$).

\subsubsection{$N_{\rm ig} = 1$:}
In the previous case, the optimal control is set to 1 in the whole interval
$[0,\tau_\alpha)$. The possibility of using one ignition process allows us
to consider controls where the system is at rest while the driver is approaching the
evader, taking advantage of the fact that when the driver is far from the evader the
circumvention term has (practically) no influence on the trajectory of the driver (and
therefore, on the behavior of the system); see Remark~2.

This delay in the ignition process substantially reduces the cost by allowing the use of
a control function profile like the one depicted in Fig.~\ref{shapes}(b), defined by
\begin{align}
 \kappa_{t_{\tt ON}}^{t_{\tt OFF}}(t) = \left\{ 
 \begin{array}{ll}
  \kappa_0 & \mbox{if } t_{\tt ON} \le t < t_{\tt OFF} \\
  0 & \mbox{elsewhere}
 \end{array},
\right.
\end{align}
where $t_{\tt ON}$ and $t_{\tt OFF}$ are the instants of time in which the control
is switched on ($\kappa = \kappa_0 = \pm1$) and off ($\kappa = 0$) respectively.
For these control functions, the cost is given by
\begin{align}
 {\cal C}(\kappa_{t_{\tt ON}}^{t_{\tt OFF}}) =
 \int_{t_0}^{t_f} \left| \kappa_{t_{\tt ON}}^{t_{\tt OFF}}(t) \right| \, dt
 = t_{\tt OFF} - t_{\tt ON},
\end{align}
so that the $(OCP)$ problem is reduced to find the values of $t_{\tt ON}$ and
$t_{\tt OFF}$ minimizing $t_{\tt OFF}-t_{\tt ON}$.

These values are found by means of two successive (numerical) shooting methods.
Let us consider the extreme case where $\rho=0$ (numerically we used $\rho=10^{-8}$),
so that the solution of the shooting method described in the previous section for
$N_{\rm ig}=0$ is unique. 

The first shooting method consists in finding, for each value of $t_{\tt ON}$, the
(unique) value $t_{\tt OFF}^*(t_{\tt ON})$ for which there exists a time $t \in (t_0,t_f)$
such that $\vec{u}_e(t) = \vec{u}_T$. This shooting method is the one presented in the
previous case where $N_{\rm ig}=0$. In fact, in that case, $t_{\tt ON}=0$.
This procedure allows us to build a function 
${t_{\tt OFF}^*}(t_{\tt ON})$ in such a way that the function
$\kappa_{t_{\tt ON}}^{t_{\tt OFF}^*}(t)$ is {\it admissible}:
$\kappa_{t_{\tt ON}}^{t_{\tt OFF}^*} \in U_{\rm ad}$.

Then, the second shooting method consists in finding the value of $t_{\tt ON}$ for which
the characteristic function of the interval $[t_{\tt ON}, t_{\tt OFF}^*(t_{\tt ON})]$
minimizes ${\cal C}(\kappa_{t_{\tt ON}}^{t_{\tt OFF}})$.

\begin{figure}[ht]
\begin{center}
       \epsfxsize=0.92\textwidth
       \epsfbox{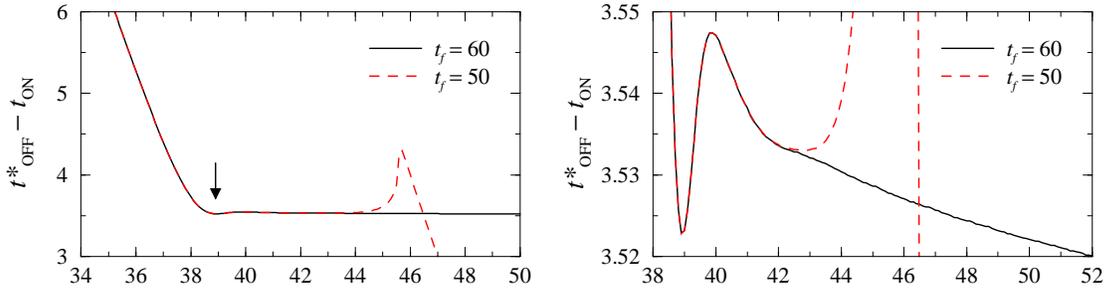} \hspace{1cm}
\vspace{-0.2cm}
\caption{Shooting method to find the optimal values $t_{\tt ON}^*$ and $t_{\tt OFF}^*$
for which $\| \vec{u}_e(t) - \vec{u}_T \| < \rho$ for some $t \in (0,t_f)$
($\rho=10^{-8}$), for two different values of $t_f$:
solid (black) line, $t_f = 60$; dashed (red) line, $t_f = 50$.
Left panel: cost function ${\cal C}(t_{\tt ON}) = t_{\tt OFF}^*(t_{\tt ON}) - t_{\tt ON}$,
exhibiting a plateau starting at $t^*_{\tt ON} \approx 38.92$ (arrow) and situated at
${\cal C}^* \approx 3.52$.
Right panel: zoom of the vertical axis, revealing the nonlinear shape of the curve with
a minimum at $t^*_{\tt ON}$, closely followed by a local maximum at
$t_{\tt ON} \approx 39.82$, and a decreasing regime for $t_f$ sufficiently large.
The relative amplitude of the nonlinearity is $\approx 7\times 10^{-3}$.}
\label{shoot-double}
\end{center}
\vspace{-0.5cm}
\end{figure}

Fig.~\ref{shoot-double} shows the cost function ${\cal C}$ as a function
of $t_{\tt ON}$: ${\cal C}(t_{\tt ON}) = t_{\tt OFF}^*(t_{\tt ON}) - t_{\tt ON}$.

When $t_f$ is sufficiently large, the cost ${\cal C}(t_{\tt ON})$ tends to a constant
value ${\cal C}^*\approx 3.52$ which is the time is takes to the driver to make the
evader turn back towards the target. See Fig.~\ref{turnbacks}, where two examples with
different ignition times $t_{\tt ON}$ yield (approximately) the same cost ${\cal C}^*$.

This value ${\cal C}^*$ constitutes a substantial reduction of the optimal cost found
for $N_{\rm ig}=0$.

In fact, Fig.~\ref{shoot-double} shows that there is a wide range of values of
$t_{\tt ON}$ yielding a similar value of the cost, so that, for a tolerance larger
than $10^{-2}$, the optimal control would not be unique. This~would allow the system
to accept another criteria or another restriction to determine the optimal strategy,
which can be for instance to minimize the total time or the total distance travelled by
the agents (and which would be equivalent to consider the cost of back-propellers).

\begin{figure}[ht]
\begin{center}
       \epsfxsize=0.95\textwidth
       \epsfbox{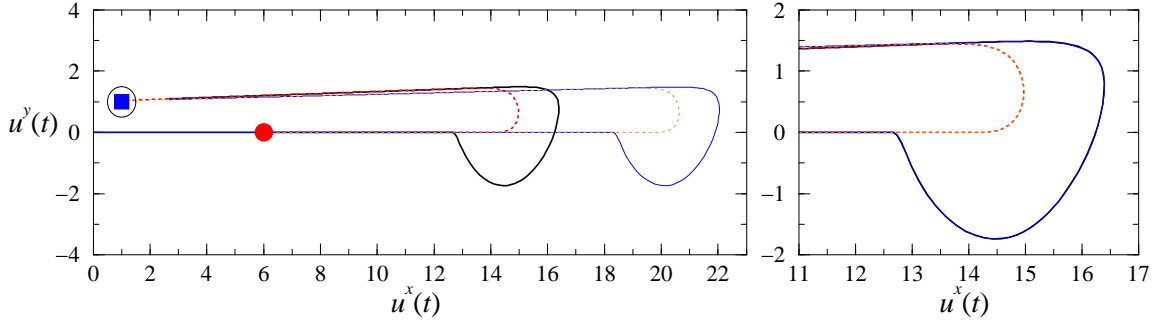}
\vspace{0.2cm}
\caption{Agents' trajectories for two {\tt ON}--{\tt OFF} controls yielding the same
cost ${\cal C}^* \approx 3.53$ in Fig.~\ref{shoot-double}.
Wide lines (short trajectories, black and red):
					$t_{\tt ON} = 42$, $t_{\tt OFF}^* = 45.5337$;
thin lines (long trajectories, blue and orange):
					$t_{\tt ON} = 50$, $t_{\tt OFF}^* = 53.5221$.
Solid lines denote the driver, dotted lines the evader.
In the short trajectory the evader reaches the target at $t_f= 65$, while it must go
until $t_f = 81$ in the long trajectory.
In both cases the turning back maneuver takes the same time ${\cal C}^*$ and has the
same shape. Right panel: trajectories during the turning back maneuver overlap when
shifted $t = t - 5.64$.} 
\label{turnbacks}
\vspace{-0.7cm}
\end{center}
\end{figure}

However, if the variations of ${\cal C}^*$ of order $10^{-3}$ are relevant, a closer
analisys shows that the first linear decreasing range in Fig.~\ref{shoot-double}, which
lasts until $t^*_{\tt ON} \approx 38.9$, is due to the fact that during this interval of
time $[0,t^*]$ the driver is far from the evader, that is,
$r = \| \vec{u}_d - \vec{u}_e \| \gg 1$, so that the term with $\kappa$ in
Eq.~(\ref{Eq-vp}) can be neglected (it is of order ${\cal O}(r^{-4})$) with respect to
the terms corresponding to the attracting force (of order ${\cal O}(r^{-2})$).
This means that switching on the control before $t^*_{\tt ON}$ doesn't contribute
effectively to reach the target and is a waste of resources. In fact, the optimal
value of the cost is reached at $t^*_{\tt ON}$: ${\cal C}(38.9)=3.5228$.
When $t_f$ is so short that the turning back maneuver can not be completed, then
$t_{\tt OFF}^*(t_{\tt ON}) = t_f$, so that the cost function ${\cal C}(t_{\tt ON})$ 
decreases (again linearly with $t_{\tt ON}$) to zero (but of course the target is not
reached). Note that, for all $t_{\tt ON}$, the value $t_{\tt OFF}^*(t_{\tt ON})$ will
exist provided $t_f$ is sufficiently large.

Interestingly, ${\cal C}(t_{\tt ON})$ exhibits an abrupt nonlinearity of small amplitude
at $t^*_{\tt ON}$; see the right panel of Fig.~\ref{shoot-double}. In a pure pursuit
regime, the velocity of the agents both converge to a constant value $v_{\rm as}$ due to
the friction with the ground. See the lower left panel in Fig.~\ref{modes}.
The nonlinearity is located precisely at the time where the evader reaches this
constant velocity ($v_{\rm as} \approx 0.71$). The curve of
${\cal C}(t_{\tt ON})$ reaches a (local) maximum at a value of $t_{\tt ON}$ slightly
larger than $t^*_{\tt ON}$, due to that the driver is close to the evader and both
agents' velocities have converged to $v_{\rm as}$.
For larger values of $t_{\tt ON}$ in the horizontal plateau of Fig.~\ref{shoot-double},
the turning back maneuver is practically identical in space and time; see
Fig.~\ref{turnbacks}, especially the right panel, where we show that two different
values of $t_{\tt ON}$ in the plateau yield two turning back maneuvers that practically
overlap.
The slight decrease of the plateau shown in the right Panel of Fig.~\ref{shoot-double}
is due to the fact that turning back maneuvers are less consuming the farther from the
target they takes place, because the angle at which the control is switched off is
smaller. This situation is reversed (that is, the plateau increases) when the target
point is above the turning back region --{\it e.g.}, at $(x,y)=(1,2)$.

Although the relative amplitude of the nonlinearity of ${\cal C}(t_{\tt ON})$ is very
small (0.024 with respect to 3.52), it unexpectedly adds an important complexity to the
study of the system: the minimum located at $t^*_{\tt ON}$ can be global or local
depending on $t_f$, therefore complicating the search for optimal directions of descent
in numerical minimization methods.

Moreover, the high sensitivity of the system detected in the previous section is also
in action if control functions like $\kappa_{t_{\tt ON}}^{t_{\tt OFF}^*}(t)$ are used
when small variations of $t^*_{\tt OFF}$ with respect to $t_{\tt ON}$ can~occur.

This section shows that, despite a quite predictive general behavior of the system (no
signs of chaos have been detected), its high sensitivity strongly suggests the use of
closed-loop controls.

\subsection{A feedback control law for $N_{\rm ig} > 1$}
\label{FeedbackControlLaw}

For systems that are subject to conditions of high sensitivity like those described in
section~3.1, closed-loop or feedback controls offer the possibility of correcting
instantaneously the state of the system for deviations from the desired
behavior~\cite{Sontag2013,Coron2007}. Moreover, the explicit form of the control as a
function of time need not to be known {\it a priori} in the whole time interval
$[t_0,t_f]$. In~turn, feedback laws have to pay the cost of continuously monitoring the
position and velocity of the agents.

We present here a feedback control law based on the following observations:
\begin{enumerate}
 \item In real situations, the orientation of the vector $\vec{v}_e(t)$ used in the
 shooting method can be difficult to observe with the accuracy required by the high
 sensitivity of the system.
 
 Instead, the alignment $a(t)$ of the driver $D$ and the
 evader $E$ with the target point $T$ is easier to observe and is a good approximation
 of $\vec{v}_e(t)$.

 \item When the driver is sufficiently far from the evader, $\kappa(t)$ can be set to
 zero (Remarks~1 \&~2).
\end{enumerate}

The instantaneous information about the state of the system is processed in real time
to determine the distance separating both agents $r(t)$ and the alignment~$a(t)$.
The alignment~$a(t)$ can be charaterized by the following scalar product
(time dependence is omitted to lighten notation):
\begin{eqnarray}
a(t) = (\vec{u}_T - \vec{u}_d) \cdot (\vec{u}_e - \vec{u}_d)^\bot
= (u_T^x - u_d^x)(u_d^y - u_e^y) + (u_T^y -u_d^y)(u_e^x-u_d^x). \quad
\end{eqnarray}
The sign of $a(t)$ reveals in which half-plane the target $T$ is with respect to the line
$\overline{DE}$, and can be used to determine the sign of $\kappa(t)$. 

Moreover, $|a(t)|$ is an instantaneous measure of how urgently the control must be set
to {\tt ON}. Let us consider a maximal tolerance of deviation $\bar{a}$. The feedback
control law is based on the idea that when $|a(t)|$ is smaller than $\bar{a}$, it is
possible to consider that $T$ is practically on the line $\overline{DE}$, so that
$\kappa(t)$ can be set to {\tt OFF}, thus saving cost, and when $|a(t)| > \bar{a}$, the
deviation is excessive and the control must be set to {\tt ON}.
The tolerance of deviation $\bar{a}$ is an effective bound for both the angle and the
intensities of the velocities
($a = \| \vec{u}_T - \vec{u}_d \| \, \| \vec{u}_e - \vec{u}_d\|
\cos (\vec{u}_T - \vec{u}_d,\vec{u}_e - \vec{u}_d)^\wedge$), so~$|a| < \bar{a}$ restricts
also the velocities of the agents: a slightly deviated evader at a high speed can miss
the target as well as a largely deviated evader at a lower speed.

Note also that when $|a(t)| < \bar{a}$, the control can be switched off provided the
evader and the target are at the same side with respect to the driver, in order to
prevent the driver from driving the evader away from the target; that is, $\kappa(t)$
can be set to zero only if the scalar product
$(\vec{u}_e - \vec{u}_T) \cdot (\vec{u}_e - \vec{u}_d)$ is negative.

Finally, Remark~2 is introduced into the feedback law by means of the characteristic function
\begin{eqnarray}
{\cal X}(t) = \left\{
\begin{array}{cl}
 0 & \mbox{ if } r^3(t) \gg \delta_2, \\
 1 & \mbox{ if not},
\end{array}
\right.
\end{eqnarray}
which serves to switch off the control when the driver is far ($r^3(t) \gg \delta_2$)
from the evader.

The feedback control law can then be written as follows:
\begin{eqnarray}
\kappa_{\rm F}(t) = {\cal X}(t) \times \left\{
\begin{array}{cl}
 0 & \mbox{ if } |a(t)| \le \bar{a}
	\mbox{ and }(\vec{u}_e - \vec{u}_T) \cdot (\vec{u}_e - \vec{u}_d) < 0, \\
 {\rm sign}\{a(t)\} & \mbox{ if } |a(t)| > \bar{a}
	\; \mbox{ or }\; \, (\vec{u}_e - \vec{u}_T) \cdot (\vec{u}_e - \vec{u}_d) \ge 0 .
\end{array}
\right.
\label{fidbaclaw}
\end{eqnarray}

We have solved the system (\ref{Eq-up})--(\ref{CondInis}) numerically using the condition
$r^3(t) > 3 \delta_2 / 2$ to have ${\cal X}(t)=0$ in expression~(\ref{fidbaclaw}). We have
considered an alignment tolerance $\bar{a} = 4 \times 10^{-1}$. The rest of values are as
in previous sections. Let us refer to this case as case (a).

The result is that the feedback law reduces substantially the cost obtained with the
open-loop control in Sec.~3.1: ${\cal C}_{\rm F}^{\rm (a)} = 1.43$, so an improvement of
60\% with respect to ${\cal C}^* = 3.53$, with however a slight increase of the number of
ignition processes, from $N^*_{\rm ig}=1$ to $N^{\rm (a)}_{\rm ig}=4$.

See Figs.~\ref{feedbk-1},~\ref{feedbk-2-3} and~\ref{feedbacks}(a1),(a2).
\begin{figure}[ht]
\begin{center}
      \epsfxsize=0.98\textwidth
      \epsfbox{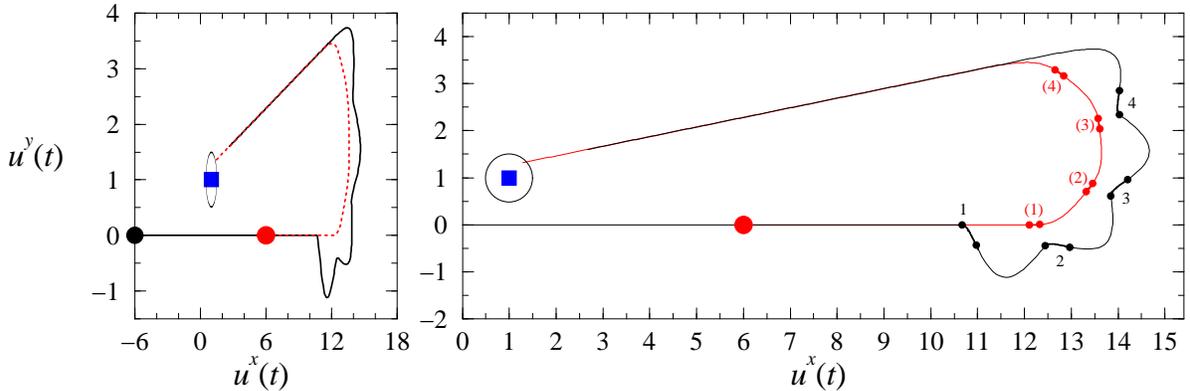}
\vspace{0.2cm}
\caption{Agents' trajectories for the feedback law $\kappa_{\rm F}(t)$ given
in~(\ref{fidbaclaw}), with alignment tolerance $\bar{a} = 4\times 10^{-1}$.
Left, whole trajectories; right, detail.
Wide solid segments denote intervals of time with active control ({\tt ON}).
Small solid circles correspond to onsets and ends of such intervals. Numbers denote
segments of the same interval of time:
$S_1 = [39.17,39.55]$, $S_2 = [41.54,41.89]$, $S_3 = [43.77,44.11]$ and $S_4 = [45.98,46.32]$. 
The cost is ${\cal C}_{\rm F}^{\rm (a)} = 1.43$.} 
\label{feedbk-1}
\vspace{-0.2cm}
\end{center}
\end{figure}

Figs.~\ref{feedbk-1} shows the trajectories of the agents (the whole trajectories in the
left panel, and a zoom of $u^x(t) \in [0,15]$ in the right panel). Surprisingly,
the trajectory of the driver exhibits an oscillatory behavior around the circular
trajectory of the evader, allowing the driver to remain closer to the evader than in
the previous cases (see, {\it e.g.}, Fig.~\ref{turnbacks}).
Moreover, the time spent with the control in active mode ($\kappa_{\rm F}=1$) is
surprisingly short compared to the time spent in this state in the open-loop control
$\kappa_{t_{\tt ON}}^{t_{\tt OFF}}(t)$.

\begin{figure}[ht]
\begin{center}
       \epsfxsize=0.95\textwidth
       \epsfbox{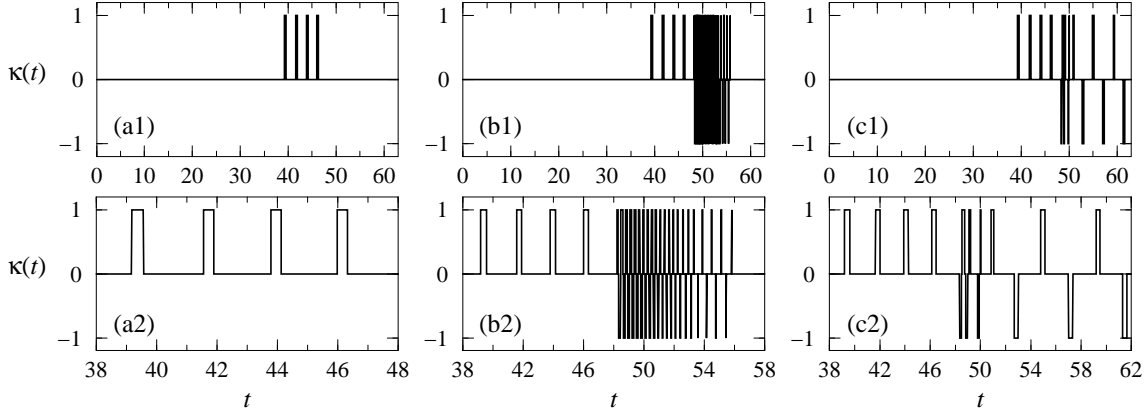}
\vspace{-.4cm}
\caption{Feedback control laws resulting from the three previous situations denoted
(a), (b) and (c) in Figs.~\ref{feedbk-1} and~\ref{feedbk-2-3}, showing the extreme
sensitivity of the
system under variations of the alignment $a^*$ and the switching times of the control
$\kappa(t)$. Lower panels (abc2): zoom of upper panels (abc1).
Resulting costs: ${\cal C}_{\rm F}^{\rm (a)} = 1.43$,
${\cal C}_{\rm F}^{\rm (b)} = 4.275$ and ${\cal C}_{\rm F}^{\rm (c)} = 4.13$.
Number of ignition processes are $N_{\rm ig}^{\rm (a)}=4$, $N_{\rm ig}^{\rm (b)}=43$
and $N_{\rm ig}^{\rm (c)}=16$.} 
\label{feedbacks}
\end{center}
\vspace{-.5cm}
\end{figure}

The improvement consists in that the circumvention mode is interrupted
($\kappa_{\rm F}=0$) during the surrounding motion of the driver around the evader,
reducing the time spent with $\kappa_{\rm F}=1$ to four small intervals $S_j$,
$j=1,\dots,4$, of total length ${\cal C}_{\rm F}^{\rm (a)} = 1.43$, and, accordingly,
$N_{\rm ig}^{\rm (a)} = 4$.
See the wide solid segments in the trajectories of the agents in Fig.~\ref{feedbk-1}
and Figs~\ref{feedbacks}(a1) and (a2), which show the resulting control function
$\kappa_{\rm F}(t)$ for $t \in [t_0,t_f]$, with $t_0=0$ and $t_f=63$.

The solution found in Fig.~\ref{feedbk-1} with the feedback law~(\ref{fidbaclaw}) can
indeed be considered a good solution of the control problem. However, this is not a
general situation, as shown by the wide range of cases analysed in our numerical
simulations, because of the high sensitivity exhibited by the system. Let us illustrate
this observation here by showing the results for two slightly different external
conditions; see the cases (b) and (c) in Figs.~\ref{feedbacks} and~\ref{feedbk-2-3}.

\begin{figure}[ht]
\begin{center}
\vspace{.7cm}
\hspace{-4.8cm}(b) \hspace{7cm} (c)
\vspace{-.7cm}
      \epsfxsize=0.9\textwidth
      \hspace{1.7cm}\epsfbox{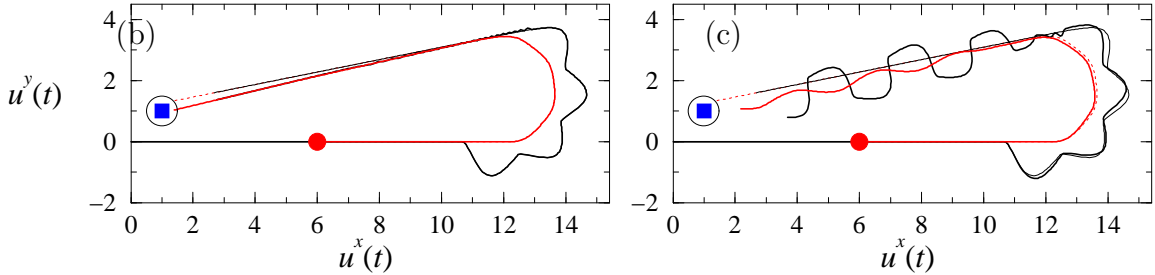}
\vspace{0.3cm}
\caption{Perturbations of the previous case depicted in Fig.~\ref{feedbk-1}:
(a) perturbation of the alignment tolerance: here $\bar{a} = 10^{-1}$ instead of
$4 \times 10^{-1}$, and
(b) perturbation of relative order $10^{-3}$ of the first interval of activation:
here $S'_1=[39.17,39.6]$ instead of $S_1=[39.17,39.55]$. For $t \ge 39.6$ the same
feedback law~(\ref{fidbaclaw}) is used.
In (c) the target is not reached because $t_f$ is too small; for a larger $t_f$, the
target is reached with no additional cost ($\kappa_{\rm F}=0$ during this extra-time).
Dashed lines represent the solution of case (a), depicted here for comparison.} 
\label{feedbk-2-3}
\end{center}
\vspace{-.2cm}
\end{figure}

Case (b) and (c) use an alignment tolerance $\bar{a}=10^{-1}$.
Fig.~\ref{feedbk-2-3} shows that the evader follows almost the same trajectory than
in case (a) (depicted in the figure to facilitate the comparison) and reaches the target
with a more accurate orientation. However, such a small deviation requires an enormous
increase of the use of the lateral propellers, as shown in Fig.~\ref{feedbacks}(b2).
Not only the cost of having the control set to 1 is larger,
${\cal C}_{\rm F}^{\rm (b)} = 4.275$
(and larger than with the open-loop controls), but a much greater number of ignition
processes is required ($N_{\rm ig}^{\rm (b)}=43$), moreover involving both the right and
the left propellers alternatively ({\it i.e.}, $\kappa_{\rm F} = +1, 0, -1, 0, +1,\dots$).

On the other hand, case (c) shows that small variations of the switching times of the
feedback control can produce huge qualitative differences in the behavior of the agents.
We have introduced a small perturbation of the first interval of activation
$S_1=[39.17,39.55]$, by keeping $\kappa = 1$ until 39.6 instead of until 39.55
(a perturbation of relative size $1-39.55/39.6 = 1.3 \times 10^{-3}$). For $t \ge 39.6$,
we use again the feedback law~(\ref{fidbaclaw}), so that a different control function
profile~arises.
The resulting behavior of the agents is depicted in Fig.~\ref{feedbk-2-3}(c).

Fig.~\ref{feedbk-2-3}(c)
shows that after the perturbation, the agents describe a widely deviated trajectory,
especially in the case of the driver, with respect to the one described with the
unperturbed feedback law shown in Fig.~\ref{feedbk-1}.
Accordingly, the perturbed control function profile, depicted in Fig.~\ref{feedbacks}(c),
is significantly different from the unperturbed feedback law.
The driver requires more time to reach the target ($t_f^{\rm c}=t_f+1=64$) than in
cases (a) and (b). Compared with case (b), the cost is higher, although not excessively,
${\cal C}_{\rm F}^{\rm (c)} = 4.13$ (with no additional cost for the extra time because
the evader moves in the right direction), and the number of ignition process is much
smaller: $N_{\rm ig}^{\rm (c)}=16$.

Perturbations of the rest of switching times produce similar results, that is,
qualitative deviations of the trajectories and a larger number of ignition processes,
with respect to the unperturbed case (a), meaning that, in cases where delays can appear
in the instants of time in which the control has to be manipulated, important changes of
the behavior of the system may arise.

The oscillatory behavior arisen in case (c) can appear to be less convenient but,
according to our numerical simulations, it is not necessarily worse than the smooth
evader-following trajectory; oscillations can contribute to reduce the number of
ignition processes, with a low increase of the cost; see left panels (a) in
Fig.~\ref{multitarget}.

\begin{figure}[ht]
\begin{center}
       \epsfxsize=0.95\textwidth
       \epsfbox{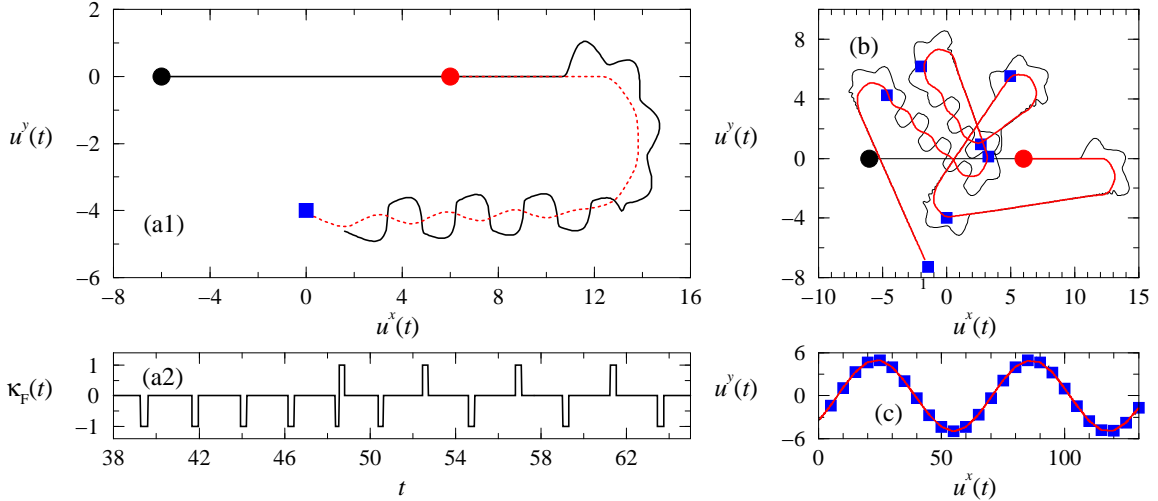}
\vspace{-.3cm}
\caption{Left panels: (a1) agents' trajectories and (a2) feedback control producing an
oscillatory behavior of the driver and a surprisingly low cost ${\cal C} \approx 4.1$.
Panel (b): agent's trajectories in a case where the path to follow is a series of 7
targets randomly distributed in a radius smaller than 8 from the original location of
the evader. Note that for some targets the trajectory is smooth ($T_1$, $T_2$, $T_3$,
$T_5$ and $T_7$) but for some others oscillations are necessary to adjust the alignment
($T_4$ and $T_6$).
Panel (c): evader's trajectory along a sinusoidal path described by a large number of
closely spaced targets.} 
\label{multitarget}
\end{center}
\end{figure}

The feedback control law is able to drive the evader along a given trajectory,
provided the trajectory is sufficiently smooth, that is, the trajectory can be described
by a series of targets; see Fig.~\ref{multitarget}(b) for a series of random targets,
and (c) for a sinusoidal trajectory described by a large number of closely spaced
targets. The study of the behavior of the system when the path to follow has a very high
curvature or describes very acute corners is matter for future work.

\section{Conclusion}
\label{IV}

We have presented an agent based model for the guidance by repulsion problem, consisting
of a system of equations corresponding to Newton's second law. The system can adopt two
operating modes that can be controlled by a single parameter $\kappa(t)$: for $\kappa=0$,
the system
is in the pure pursuit mode (control in mode ${\tt OFF}$, {\it i.e.}, resting state), and
for $\kappa=1$, the system is in the circumvention mode (control in mode ${\tt ON}$,
{\it i.e.}, active state). We have shown that by appropriately defining the function
$\kappa(t): [t_0,t_f] \to \{-1,0,1\}$, the driver can guide the evader to any desired
target or along any (relatively smooth) path. We have then formulated an optimal control
problem $(OCP)$ to find the optimal guidance strategy minimizing the cost in terms of the
number of times the system is activated from its resting state and the time the system
spends in the active mode. By means of (numerical) shooting methods, we have obtained
the optimal open-loop strategies for the case where the number of activations is equal
or smaller than one, finding that the system is highly sensitive to small variations of
the activation/deactivation times.

These results show that open-loop controls would
not be of practical interest in real problems due to the presence of external
perturbations, thus suggesting the use of a feedback law.
Taking advantage of the information provided in the study of the behavior of the system
under open-loop controls for $N_{\rm ig} \le 1$, we have designed a feedback law for the
case where $N_{\rm ig} > 1$, which allows to correct in real time for deviations from the
desired trajectory. We have found that the feedback law is also highly sensitive to small
variations of the conditions of the problem, in this case, of the accuracy with which the
target is reached ({\it i.e.}, the radius of the target ball). Moreover, our results
show that the feedback law, and therefore the resulting behavior of the system, are
highly sensitive to possible delays in the switching times of the control. This means
that, in systems or devices where the manipulation of the control cannot be carried out
at arbitrarily close instants of time, the behavior of the driver can exhibit large
oscillations that can produce an increase of the cost. This may happen in situations
where time delays exist in collecting and interpreting the data about the state of the
system or in the reaction time of the system once the control is changed, especially when
two consecutive changes are very close in time.

The main direction for the immediate future work consists in taking into account the cost
for the driver to get close to the evader, both in time and travelled distance. This
would correspond to add the cost of back propellers to the cost functional:
\begin{align}
 J_B(t_f) = \eta_1 \int_{t_0}^{t_f} \| \vec{u}_d(t) - \vec{u}_d(t_0)\|^2 \,dt + \eta_2 t_f,
\end{align}
and find a feedback control law for the two controls $\kappa(t)$ and $t_f$.

The interest of guidance by repulsion could also be extended to the case where the
evader's behavior has a stochastic component and when multiple agents (evaders and/or
drivers) are considered. The feedback law will be especially relevant when noise is
considered in both the behavior of the agents and in the manipulation of the data.


\section*{Acknowledgments}
We thank the anonymous reviewers for their interesting suggestions and for pointing out
the interest of open problem which include the cost of back propellers in the cost
functional.
This material is based upon work that has received funding from
the European Union’s Horizon 2020 research and innovation programme under
the Marie Sk\l{}odowska-Curie grant agreement No 655235, entitled ``SmartMass'', 
the Advanced Grant NUMERI-WAVES/FP7-246775 of the European Research Council Executive Agency,
the BERC 2014-2017 program and the pre-doctoral grant PRE-2014-1-461 of the Basque Government,
the FA9550-15-1-0027 of AFOSR,
the MTM2014-52347 and MTM2011-29306 Grants and the Severo Ochoa program SEV-2013-0323 of
the MINECO,
and a Humboldt Award at the University of Erlangen-N\"urnberg.

\appendix
\addtocontents{toc}{\setcounter{tocdepth}{0}}

\section{Asymptotic value of the velocities at long times when $\kappa = 0$}
\label{Long0}
Numerical simulations show that when $\kappa(t) = 0$ continuously for a sufficiently
large time, both agent's velocities converge asymptotically to the same constant velocity
$\vec{v}_{\rm sat}$. In that state, $\dot{\vec{v}}_d(t) = \dot{\vec{v}}_e(t) = 0$ and,
from Eqs.~(\ref{Eq-vp-d})--(\ref{Eq-ve-d}), we have
\begin{align}
- {C^E_D \over \nu_d}{\vec{u}_d - \vec{u}_e \over \delta_{\rm as}^2}
 \left[1 + {1 \over \delta_{\rm as}^2} 
 \left({C_{\rm R} \over C^E_D} \delta_1^4 - \delta_c^2 \right) \right]
  = - {C^D_E \over \nu_e} {\vec{u}_d - \vec{u}_e \over \delta_{\rm as}^2} = \vec{v}_{\rm as}.
\end{align}
Thus, comparing norms, we obtain $v_{\rm as} = C^D_E/(\nu_e \, \delta_{\rm as})$, and
extracting $\delta_{\rm as}$, we obtain
\begin{align}
 \delta_{\rm as} = \sqrt{ \nu_e(C^E_D \delta_c^2 - C_{\rm R} \delta_1^4)
		\over \nu_e C^E_D - \nu_d C^D_E \,},
\label{radis}
\end{align}
provided $\nu_e C^E_D>\nu_d C^D_E$, as it is the case for the values we are considering.
Note also that a necessary condition to have an effective short-range repulsion acting on
the driver is that the factor between large parentheses in Eq.~(\ref{Eq-vp-d}) is
positive, so $C^E_D \delta_c^2 - C_{\rm R} \delta_1^4 > 0$ and the radicand
in expression~(\ref{radis}) is positive.

\section{On the controllability of the driver-evader system}
\label{Lyap}
We show here that driver and evader are prevented from separating infinitely from each
other and tend asymptotically to be separated a distance of order one. The proof follows
the idea of ``free agents'' used in \cite{Shi-Xie2011} (see also \cite{EscobedoEtAl2014}
for a more similar model).
 
\begin{definition}
A driver agent $D$ is said to be a {\it free agent} at time $t$ if its distance to the
evader $E$ is greater than an arbitrarily large positive constant $\delta \gg 1$; that is,
$r(t) = \|\vec{u}_d(t) - \vec{u}_e(t)\| \geq \delta$.
\end{definition}

\begin{proposition}
If the driver $D$ is a free agent, then the system~(\ref{Eq-up})--(\ref{Eq-vs})
can be reduced as follows:
\begin{align}
\dot{\vec{u}}_d(t) & =  \vec{v}_d(t), \\
\dot{\vec{u}}_e(t) & =  \vec{v}_e(t) ,\\
\dot{\vec{v}}_d(t) & =  -  {C^E_D \over m_d}
 {\vec{u}_d(t) - \vec{u}_e(t) \over \|\vec{u}_d(t)-\vec{u}_e(t)\|^2}
  - {\nu_d \over m_d} \vec{v}_d(t),
\label{A-dvp}\\
\dot{\vec{v}}_e(t) & =  {C^D_E \over m_e}
 {\vec{u}_e(t) - \vec{u}_d(t) \over \|\vec{u}_e(t)-\vec{u}_d(t)\|^2}
  - {\nu_e \over m_e} \vec{v}_e(t).
\label{A-dve}
\end{align}
\end{proposition}

\begin{proof}
Using Remark~1.
\end{proof}

\begin{lemma} If $m_d/\nu_d < m_e/\nu_e$, then, $\forall t >0$,
\begin{align}
\norm{\vec{v}_e(t)}^2 \le \frac{C_E^D}{C_D^E} \frac{m_d}{m_e} \norm{\vec{v}_d(t)}^2
\Longrightarrow
\norm{\vec{v}_e(t)}^2 \le \frac{C_E^D}{C_D^E} \frac{\nu_d}{\nu_e} \norm{\vec{v}_d(t)}^2.
\end{align}
\label{lema1}
\end{lemma}

\begin{definition}
Let ${\bf q} = (\vec{u}_d, \vec{v}_d,\vec{u}_e, \vec{v}_e) \in \mathbb{R}^8$
and $V({\bf q})$: $\mathbb{R}^8 \rightarrow  \mathbb{R} $ be the following
potential functional:
\begin{align}
V({\bf q}) & =  \ln (\norm{\vec{u}_d-\vec{u}_e}) + \frac{1}{2} \frac{m_d}{C^E_D} \norm{\vec{v}_d}^2
-  \frac{1}{2} \frac{m_e}{C^D_E} \norm{\vec{v}_e}^2.
\end{align}
Then, for free agents ({\it i.e.}, $r(t) \ge \delta$) and under the hypotheses of Lemma~\ref{lema1},
$V({\bf q})$ is positive.
\end{definition}

\begin{theorem}
If the driver $D$ is a free agent, then $V({\bf q})$ is bounded from below and
$\dot{V}({\bf q})$ is negative along the agents' trajectories defined by the
system~(\ref{Eq-up})--(\ref{Eq-vs}). Consequently, $V({\bf q}(t))$
converges in time to a minimum which is reached when the distance between both agents
is~$\delta$.
\label{teolyapunov}
\end{theorem}

\begin{proof}
The time-derivative of $V({\bf q})$ along the agents' trajectories is given by:
\begin{align}
\dot{V}({\bf q}) =  \nabla_{\vec{u}_d} V \cdot \frac{d \vec{u}_d}{dt}
+ \nabla_{\vec{v}_d} V \cdot \frac{d \vec{v}_d}{dt} 
+ \nabla_{\vec{u}_e} V \cdot \frac{d \vec{u}_e}{dt}
+ \nabla_{\vec{v}_e} V \cdot \frac{d \vec{v}_e}{dt},
\qquad \\
\mbox{where} \quad
\nabla_{\vec{u}_d} V = - \nabla_{\vec{u}_e} V
= {\vec{u}_d - \vec{u}_e \over \|\vec{u}_d-\vec{u}_e\|^2}, \quad
\nabla_{\vec{v}_d} V = \frac{m_d}{C^E_D} \vec{v}_d \quad \mbox{and} \quad
\nabla_{\vec{v}_e} V = - \frac{m_e}{C^D_E} \vec{v}_e.
\end{align}
Then:
\begin{align}
\dot{V}({\bf q}) =
& \quad {\vec{u}_d - \vec{u}_e \over \|\vec{u}_d-\vec{u}_e\|^2} \cdot \vec{v}_d
+ \frac{m_d}{C^E_D} \vec{v}_d \cdot
  \left( - \frac{C^E_D}{m_d} {\vec{u}_d - \vec{u}_e \over \|\vec{u}_d-\vec{u}_e\|^2}
  - \frac{\nu_d}{m_d} \vec{v}_d\right)
\\
& - {\vec{u}_d - \vec{u}_e \over \|\vec{u}_d-\vec{u}_e\|^2} \cdot \vec{v}_e
- \frac{m_e}{C_E^D} \vec{v}_e \cdot
  \left( - \frac{C^D_E}{m_e} {\vec{u}_d - \vec{u}_e \over \|\vec{u}_d-\vec{u}_e\|^2}
  - \frac{\nu_e}{m_e} \vec{v}_e\right)
\\
= & - \frac{\nu_d}{C^E_D} \| \vec{v}_d \|^2 + \frac{\nu_e}{C^D_E} \| \vec{v}_e \|^2,
\end{align}
which, under the conditions of Lemma~\ref{lema1}, is negative.
Then $V({\bf q})$ decreases and is bounded from below, so $V({\bf q})$ has a minimum,
which is reached when $\| \vec{v}_d(t) \| = \| \vec{v}_e(t) \| = 0$ and
$r(t) = \| \vec{u}_d(t) - \vec{u}_e(t) \| = \delta$ (which is the minimum value of
$r(t)$ for a free agent).
\end{proof}

Thus, agents are prevented from separating infinitely from each other because as soon as
$r(t) \ge \delta$, the driver becomes a free agent and is forced to move back towards the
evader, provided the balance between the mass and the friction of the agents verifies
Lemma~\ref{lema1}.

\section{Simple shooting method}
\label{Shoot}

Once the evader is turning back towards the target, that is, $v_e^x(t) <0$, we check
the direction of the velocity vector of the evader with the line $\overline{ET}$
described by the evader and the target. Then, if for some time $t \in (t_0,t_f)$ the
vector $\vec{v}_e(t)$ points towards a point located below $T$, the tentative value
of $t_{\tt OFF}$ must be reduced; if instead, at the final time $t_f$, the vector
$\vec{v}_e(t)$ points towards a point located above $T$, then the tentative value
of $t_{\tt OFF}$ must be augmented.

That is: Given a value of $t_{\tt ON}$, take an initial value of
$t_{\tt OFF}$ larger than $t_{\tt ON}$ and:
\begin{enumerate}
 \item[{\tt 0.}] Solve the system (\ref{Eq-up})--(\ref{CondInis}) with
 $\kappa_{t_{\tt ON}}^{t_{\tt OFF}}(t)$ for $t \in (t_0,t_f)$ and anotate the value of
 $t_b$, which is the first time such that $v_e^x(t_b)<0$. If no such time is
 reached, this means that $t_{\tt OFF}<t_b$, so take a larger value of $t_{\tt OFF}$ and
 shoot again ({\tt goto~0}).

 \item[{\tt 1.}] For each time $t \in (t_b,t_f)$, evaluate the instantaneous alignment
 $\alpha(t)$ of the velocity vector of the evader $\vec{v}_e(t)$ with respect to the
 target point $\vec{u}_T$:
\begin{align}
 \alpha(t) = \big(u_T^y - u_e^y(t)\big) v_e^x(t)
           - \big(u_T^x - u_e^x(t)\big) v_e^y(t).
\end{align}
 Then, if $\alpha(t) < 0$, take a smaller value of $t_{\tt OFF}$ and shoot again ({\tt goto~0}).

 \item[{\tt 2.}] If at time $t_f$ the velocity vector of the evader is still pointing above
 the target, that is, $\alpha(t_f) > 0$, then take a larger value of $t_{\tt OFF}$ and
 shoot again ({\tt goto~0}).
\end{enumerate}
The new value of $t_{\tt OFF}$ for the next shoot can be selected with a simple method
({\it e.g.}, bisection).\\
{\tt Stop} when $|\alpha(t)| < \epsilon$, for a small value of the tolerance $\epsilon$;
the value of $t^*_{\tt OFF}(t_{\tt ON})$ has been found, proceed to the next value of
$t_{\tt ON}$.




\end{document}